\documentclass[12pt]{amsart}

\usepackage{epsfig,amsmath,amsthm,amssymb,graphicx,enumerate}
\usepackage[top=1in, bottom=1in, left=1in, right=1in]{geometry}
\usepackage{lineno}

\newtheorem{prop}{Proposition}[section]
\newtheorem{defn}[prop]{Definition}
\newtheorem{thm}[prop]{Theorem}
\newtheorem{cor}[prop]{Corollary}
\newtheorem{lem}[prop]{Lemma}
\newtheorem*{conj_a}{Conjecture}
\newtheorem*{defn_1}{Definition 2.1}
\newtheorem*{prop_1}{Proposition}
\newtheorem*{cor_1}{Corollary}
\newtheorem*{thm_main}{Main Theorem}
\newcommand{\quot}[2]{{\raisebox{.2em}{$#1$}\left/\raisebox{-.2em}{$#2$}\right.}}
\newcommand{\lk}{\text{lk}}
\def\mathbi#1{\textmd{\em #1}}
\newcommand{\Bl}{\mathcal{B}\mathbi{l}}
\newcommand{\integers}{\mathbb{Z}}
\begin{document}
\title[Slice knots which bound punctured Klein bottles]{Slice knots which bound punctured Klein bottles}

\author{Arunima Ray}
\address{Department of Mathematics, Rice University}
\email{arunima.ray@rice.edu}
\urladdr{www.math.rice.edu/$\sim$ar25}

\date{\today}

\subjclass[2000]{57M25}

\begin{abstract}
We investigate the properties of knots in $\mathbb{S}^3$ which bound punctured Klein bottles, such that a pushoff of the knot has zero linking number with the knot, i.e.\ has \textit{zero framing}. This is motivated by the many results in the literature regarding slice knots of genus one, for example, the existence of homologically essential zero self-linking simple closed curves on genus one Seifert surfaces for algebraically slice knots. Given a knot $K$ bounding a punctured Klein bottle $F$ with zero framing, we show that $J$, the core of the orientation-preserving band in any disk-band form of $F$, has zero self-linking. We prove that such a $K$ is slice in a $\mathbb{Z}\left[\frac{1}{2}\right]$-homology $\mathbb{B}^4$ if and only if $J$ is as well, a stronger result than what is currently known for genus one slice knots. As an application, we prove that given knots $K$ and $J$ and any odd integer $p$, the $(2,p)$ cables of $K$ and $J$ are $\mathbb{Z}\left[\frac{1}{2}\right]$-concordant if and only if $K$ and $J$ are $\mathbb{Z}\left[\frac{1}{2}\right]$-concordant. In particular, if the $(2,1)$-cable of a knot $K$ is slice, $K$ is slice in a $\mathbb{Z}\left[\frac{1}{2}\right]$-homology ball. 
\end{abstract}

\maketitle

\section{Introduction}

A \textit{knot} is the image of a smooth embedding $\mathbb{S}^1 \hookrightarrow \mathbb{S}^3 = \partial\mathbb{B}^4$. A knot is called \textit{slice} if it bounds a smoothly embedded disk in $\mathbb{B}^4$. The set of knots, modulo slice knots, under the connected sum operation forms an abelian group called the \textit{knot concordance group}, denoted by $\mathcal{C}$. In \cite{Lev69a,Lev69b}, Levine described a surjection from $\mathcal{C}$ to $\integers^\infty \oplus\, \left(\integers/2\integers\right)^\infty \oplus\, \left(\integers/4\integers\right)^\infty$. Knots in the kernel of this map are said to be \textit{algebraically slice}. The quotient of $\mathcal{C}$ by algebraically slice knots is called the \textit{algebraic knot concordance group}, denoted $\mathcal{AC}$.

It is a well-known fact that given any knot $K$, we can find an embedded oriented surface in $\mathbb{S}^3$ whose single boundary component is $K$. Such a surface is called a \textit{Seifert surface}. Seifert surfaces give rise to a multitude of knot invariants, such as the \textit{genus} of $K$, the minimum genus of a Seifert surface for $K$. There are many results in the literature about the properties of genus one knots, i.e.\ knots which bound punctured tori. These represent the simplest non-trivial class of Seifert surfaces. In \cite{Gi83} Gilmer showed that if a knot $K$ is algebraically slice and bounds a punctured torus $F$ then, up to isotopy and orientation, there are exactly two homologically essential simple closed curves $J_1$ and $J_2$ on $F$ with zero self-linking with respect to the Seifert form on $F$. This is an important result, since if one of these curves $J_i$ is a slice knot, $K$ must be slice as well since we can construct a slice disk or $K$ by surgering $F$ along $J_1$ or $J_2$. Consequently, curves such as $J_1$ and $J_2$, namely, homologically essential simple closed curves on a genus one surface with self-linking zero, are called \textit{surgery curves} for $F$. In 1982 \cite[Strong Conjecture, p. 226]{Kauff87}, Kauffman conjectured the converse as follows:

\begin{conj_a}[Kauffman's Conjecture, Problem N1.52 of \cite{Kirby84}]$K$ is a slice knot with a genus one Seifert surface $F$ if and only if $F$ has a surgery curve which is slice.\end{conj_a}
 
Much work has been done towards proving this result \cite{CHL10, COT03, Coo82, Gi93}. Casson-Gordon theory can be used to show that at least one of the curves $J_i$ must satisfy some strong requirements on its algebraic concordance class, but it was recently shown that these fail to imply a vanishing signature function \cite{GiLiv11}. Soon after this present paper was completed, Cochran and Davis \cite{CD13} showed that Kauffman's conjecture is false. In particular, they constructed (smoothly) slice knots that admit Seifert surfaces such that neither surgery curve has zero Arf invariant. Moreover, there exist examples where the Seifert surfaces considered are the unique minimal genus Seifert surface up to isotopy. We note that the Arf invariant is a remarkably weak invariant and therefore, Cochran and Davis have shown that very little can be said about the concordance properties of surgery curves on genus one Seifert surfaces for slice knots. 

The motivation for this paper is to understand what is true in the analogous context of knots which bound punctured Klein bottles. (Notice that this is slight abuse of terminology: we are referring to Klein bottles with a disk removed, whose single boundary component consists of the knot. These are of course different from \textit{punctured} Klein bottles, but we retain the terminology for the sake of brevity.) Recall that for a connected, compact non-orientable surface the term \textit{genus} is used to refer to the number of summands in its unique decomposition as a connected sum of real projective planes (with disks removed if necessary). In \cite{Cl78}, Clark defined the \textit{crosscap number} of a knot, denoted $c(K)$, to be the minimum genus of non-orientable surfaces bounded by $K$. This invariant is occasionally referred to as the \textit{crosscap genus} or the \textit{non-orientable genus} of $K$. $c(K)$ is a useful invariant since there are knots of arbitrarily large genus with $c(K)=1$. A lot of work has been done on computing the crosscap numbers of certain families of knots, such as in \cite{Cl78,HiraTera06,IchiMizu10,MuraYasu95,Tera01,Tera04}. 

Knots with $c(K)=1$ are completely classified by the following result of Clark: 

\begin{prop_1}[Proposition 2.2. from \cite{Cl78}]\label{c(k)=1 is cable} $c(K)=1$ if and only if $K$ is a $(2,n)$--cable knot. \end{prop_1}

As a result, punctured Klein bottles represent the simplest classes of non-orientable surfaces bounded by knots which are not easily understood. Knots bounding punctured Klein bottles were used in \cite{HeLivRu10} to construct examples of topologically slice knots with nontrivial Alexander polynomials. 

Suppose a knot $K$ bounds a non-orientable surface $F$. If we define the \textit{longitude} $\lambda$ of $K$ to be a pushoff in the direction of $F$, we see that $\lambda$ bounds a non-orientable surface in the knot complement and therefore, has even linking number with the knot. In this paper we will often assume, to parallel the orientable case, that $\lk(K,\lambda)=0$

\begin{defn_1}Let $K\subseteq \mathbb{S}^3$ be a knot and $F\subseteq \mathbb{S}^3$ be a non-orientable surface with $K = \partial F$. Let $N(K)$ be a regular neighborhood of $K$. We refer to $\lambda = F \cap  \partial N(K)$ as the \textit{longitude} of $K$. We define the \textit{framing} of $F$ to be $\lk(K,\lambda)$, denoted $\mathcal{F}(F)$.\end{defn_1}

The main result of our paper is the following:

\begin{thm_main} If a slice knot $K$ bounds a punctured Klein bottle $F$ with $\mathcal{F}(F)=0$, we can find a 2-sided homologically essential closed curve $J$ embedded in $F$ with self-linking zero which is slice in a $\mathbb{Z}\left[\frac{1}{2}\right]$-homology ball and hence, rationally slice (i.e.\ slice in a $\mathbb{Q}$-homology $\mathbb{B}^4$).\end{thm_main}

We will see that surgering $F$ along a slice curve $J$ as mentioned in the above theorem also yields a slice disk for $K$. Therefore, the notion of surgery curve can be extended to non-orientable surfaces of genus 2. 

Rational concordance has been studied extensively and in great generality \cite{Cha07}. Being rationally slice is a strong condition since many classical concordance invariants secretly obstruct knots being $\mathbb{Q}$-concordant. For example, it is known that both the Levine–Tristram signature function and the $\tau$-invariant of Ozsváth and Szabó and Rasmussen \cite{OS03, Ras03} are zero for rationally slice knots. Therefore, in marked contrast to the genus one case, our result shows that there are very strong restrictions on the concordance class of surgery curves on punctured Klein bottles.

We will start this paper by proving some general properties of non-orientable surfaces bounded by knots with zero framing, followed by our main theorem and other results relating to concordance. The tools developed will enable us to prove a surprising corollary about cable knots. We will use the notation $K_{(m,n)}$ to denote the $(m,n)$--cable of a knot $K$. Details about the cabling operation can be found in any introductory knot theory textbook, such as \cite[Chapter 4D]{Ro90}. It can be easily shown that given concordant knots $K$ and $J$, $K_{(m,n)}$ and $J_{(m,n)}$ are concordant for any choice of $m$ and $n$. Using our results in Sections 2 and 3 we will prove the following partial converse.

\begin{cor_1}Given knots $K$ and $J$ and any odd integer $p$, if $K_{(2,p)}$ is concordant to $J_{(2,p)}$ then $K$ is concordant to $J$ in a $\mathbb{Z}\left[\frac{1}{2}\right]$-homology $\mathbb{S}^3\times [0,1]$. In particular, if $K_{(2,p)}$ is concordant to the $(2,p)$--torus knot, then $K$ is slice in a $\mathbb{Z}\left[\frac{1}{2}\right]$-homology $\mathbb{B}^4$.\end{cor_1}

This result is related to the recent work studying whether satellite operations are injective on the smooth knot concordance group \cite{CHL11,HeK10}, i.e.\ if two satellite knots on the same pattern knot are concordant, are the companion knots concordant? The (conjectured) smooth injectivity of the Whitehead doubling operator, for instance, has been studied for many years \cite[Problem 1.38]{kirbylist}. Corollary \ref{cableconv} has been generalized by Cochran, Davis and the author to a much larger family of satellite operators in \cite{CDR12}. 

\subsection{Notation and definitions}
We will work in the smooth category. Two knots $K_i\hookrightarrow\mathbb{S}^3=\partial \mathbb{B}^4$, $i=0,1$, are said to be \textit{concordant} if there exists a smooth proper embedding of an annulus into $\mathbb{S}^3\times [0,1]$ that restricts to $K_i$ on each $\mathbb{S}^3\times \{i\}$. A knot is called \textit{slice} if it is concordant to the unknot, or equivalently, if it is the boundary of a smooth embedding of a 2-disk in $\mathbb{B}^4$. 

There is a corresponding notion of knots being slice and concordant in spaces which look like $\mathbb{B}^4$ and $\mathbb{S}^3\times[0,1]$ with respect to homology with specified coefficients. Suppose $R\subseteq \mathbb{Q}$ is a non-zero subring. A space $X$ is called an \textit{R-homology Y} if $H_*(X;R) \cong H_*(Y;R)$. Knots $K_0$ and $K_1$ in $\mathbb{S}^3$ are said to be $R$-\textit{concordant} if there exists a compact, oriented, smooth 4-manifold $W$ such that $W$ is an $R$-homology $\mathbb{S}^3 \times [0,1]$, $\partial W = \mathbb{S}^3\times \{0\}\, \sqcup\, -\mathbb{S}^3\times \{1\}$, and there exists a smooth properly embedded annulus in $W$ which restricts on its boundary to the given knots. We say that $K$ is \textit{$R$-slice} if it is $R$-concordant to the unknot, or equivalently if it bounds a smoothly embedded 2-disk in an $R$-homology 4-ball whose boundary is $\mathbb{S}^3$. The set of knots modulo $R$-slice knots forms an abelian group.

Two 3-manifolds $M_1$ and $M_2$ are said to be \textit{homology cobordant} if there exists a 4-manifold $W$ which is a smooth cobordism between $M_1$ and $M_2$, such that $H_*(W,M_1)=0=H_*(W,M_2)$. For $R$ as above, $M_1$ and $M_2$ are called \textit{$R$-homology cobordant} if there exists a $W$ as above with the weaker requirement that $H_*(W,M_1;R)=0=H_*(W,M_2;R)$. For any knot $K$ we will use the notation $M_K$ to denote the zero-framed surgery on $K$.

A curve $\gamma$ on a surface $F\subseteq \mathbb{S}^3$ is called \textit{2-sided} if it has a regular neighborhood in $F$ homeomorphic to an annulus, i.e.\ it has a trivial normal bundle. It is well known that $\gamma$ is orientation-preserving iff it is 2-sided. A 2-sided $\gamma$ has a regular neighborhood which is an annulus. Let $\gamma^+$ and $\gamma^-$ denote the two boundary components of this annulus. The \textit{self-linking} of $\gamma$ is defined to be $\lk(\gamma,\gamma^+)=\lk(\gamma^-,\gamma)=\lk(\gamma^-,\gamma^+)$. 

We will also frequently require the `disk-band' form of an embedded surface with boundary. We recall that given any embedding in $\mathbb{S}^3$ of a surface $F$ with a single boundary component, there is an ambient isotopy of $\mathbb{S}^3$ taking $F$ to the standard form of a disk with bands attached, wherein the bands may be twisted, linked or knotted, by collapsing towards the 1-skeleton. This process is described in \cite[pp. 81]{Kauff87}. We will additionally require that the disk-band form of a punctured Klein bottle contain an orientation preserving band, i.e.\ exactly one of the two bands in the disk band form has an odd number of half-twists. 

\section{Properties of knots bounding punctured Klein bottles with zero framing}

We recall the following definition from Section 1: 
\begin{defn}\label{framingdef}Let $K\subseteq \mathbb{S}^3$ be a knot and $F\subseteq \mathbb{S}^3$ be a non-orientable surface with $K = \partial F$. Let $N(K)$ be a regular neighborhood of $K$. We refer to $\lambda = F \cap  \partial N(K)$ as the \textit{longitude} of $K$. We define the \textit{framing} of $F$ to be $\lk(K,\lambda)$, denoted $\mathcal{F}(F)$.\end{defn}

Given an embeddng of a surface $F$, we can first perform an ambient isotopy on $\mathbb{S}^3$ to get $F$ in disk-band form. Given such an embedding, one can obtain $\mathcal{F}(F)$ by drawing a parallel to the boundary and computing the linking number. Such a calculation can be performed solely on the basis of the types and numbers of crossings of the bands and the twists in each band. 

We notice that $\lambda$ bounds a non-orientable surface in the complement of $K$, and therefore, $\mathcal{F}(F)$ is always an even number. In this paper we will often further restrict $\mathcal{F}(F)$ to be zero to mirror the orientable case. We start by investigating some implications of the zero framing condition on any non-orientable surfaces which bound knots. First of all, it would be nice to know that this is possible: 
\begin{figure}[t]
  \centering
  \includegraphics{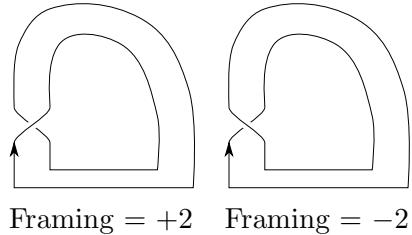}
  \put(-2.1,-0.2){\small Framing = $+2$}
  \put(-0.97,-0.2){\small Framing = $-2$}
  \caption[M\"{o}bius bands bounded by the unknot. Given a non-orientable surface $F$ bounded by some knot $K$, we boundary connect sum $F$ with the M\"{o}bius bands above to change framing without changing the knot type of the boundary.]{M\"{o}bius bands bounded by the unknot. Given a non-orientable surface $F$ bounded by some knot $K$, we boundary connect sum $F$ with the M\"{o}bius bands above to change framing without changing the knot type of the boundary.}\label{fixframing}
\end{figure}
\begin{prop}Any knot $K$ bounds some non-orientable surface (compact with a single boundary component smoothly embedded in $\mathbb{S}^3$) with zero framing.\end{prop}
\begin{proof}We know that any knot $K$ bounds some such non-orientable surface, $F$, obtained using the checkerboard coloring of a diagram for $K$ \cite{Cl78}. $\mathcal{F}(F)$ is an even number, which is additive under boundary connect sum of surfaces by the remarks above. We can take the boundary connect sum of $F$ with as many copies of the M\"{o}bius bands in Figure \ref{fixframing} to change the framing to $0$. This does not change the knot type of the boundary since the M\"{o}bius bands in Figure \ref{fixframing} bound unknots.

Alternately, one can start with a Seifert Surface $F$ for $K$, and then boundary connect sum $F$ with both the M\"{o}bius bands in Figure \ref{fixframing}. The resulting framing will be $0+2-2=0$\footnote{We are grateful to the anonymous referee who suggested this alternate proof.}.\end{proof}

\begin{lem}\label{genus-framing}Suppose a knot $K$ bounds a non-orientable surface $F$ (with a single boundary component) with framing $\mathcal{F}(F)$. $\mathcal{F}(F)\equiv 2 \mod{4}$ if and only if the genus of $F$ is odd; $\mathcal{F}(F)\equiv 0 \mod{4}$ if and only if the genus of $F$ is even. In particular, if $\mathcal{F}(F)=0$, $F$ has even genus. \end{lem}

\begin{proof}For any knot $K$, if $F$ is a surface (possibly non-orientable) with $\partial F = K$, there exists a non-singular symmetric bilinear form \cite{GLi78}\cite[Chapter 9]{Lick97} $$\mathcal{G}_F:H_1(F)\times H_1(F) \rightarrow \mathbb{Z}$$ such that $$\sigma(K) = \text{sign}(\mathcal{G}_F)-\frac{1}{2}\mathcal{F}(K)$$ where $\text{sign}(\mathcal{G}_F)$ is the signature of any matrix representing the bilinear form $\mathcal{G}_F$. Since $\mathcal{G}_F$ is a non-singular bilinear form on $H_1(F)$, $\text{sign}(\mathcal{G}_F)$ is even exactly when $\text{dim}(H_1(F))$ is even, i.e.\ $F$ has even genus. Since $\sigma(K)$ is always even, $\frac{1}{2}\mathcal{F}(F)$ is even exactly when $F$ has even genus. \end{proof}

\begin{figure}[t]
  \centering
  \includegraphics{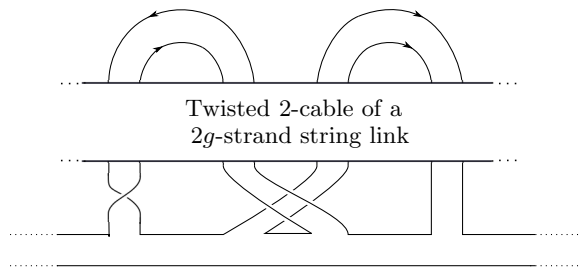}
  \put(-2.07,0.8){\scriptsize Twisted $2$-cable of a}
  \put(-2.04,0.66){\scriptsize $2g$-strand string link}
  \caption[Disk band picture of a general non-orientable surface with even genus $g$.]{Disk band picture of a general non-orientable surface with even genus $g$.}\label{nosurface}
\end{figure}
\begin{prop}\label{||curve}If a knot $K$ bounds a non-orientable surface $F$ (with a single boundary component) with zero framing, there is a 2-sided homologically essential closed curve embedded in $F$ with zero self-linking. In particular, the curve constructed is the Poincar\'e dual to $w_1(F)$, the first Stiefel-Whitney class of the tangent bundle of $F$.\end{prop}
\begin{figure}[b]
  \centering
  \includegraphics{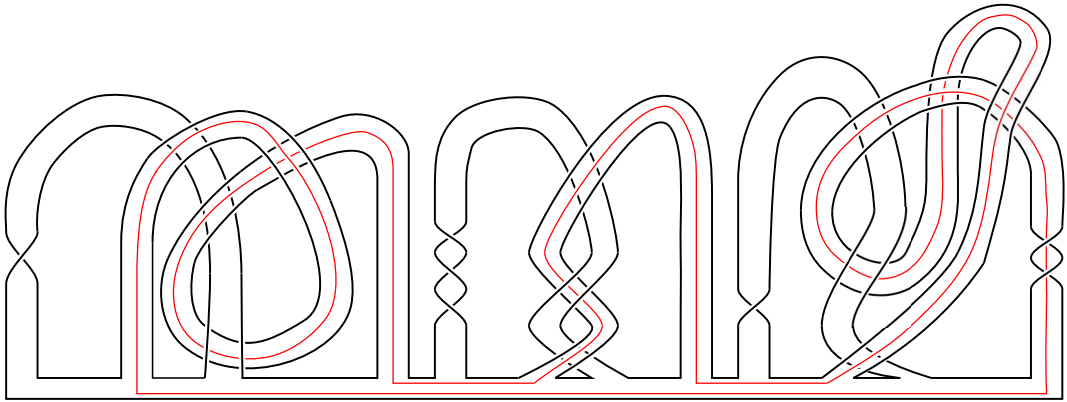}
  \caption[Curve of self-linking zero on a non-orientable surface with zero framing.]{Curve of self-linking zero on a non-orientable surface bounding a knot with zero framing.}\label{examplecurve}
\end{figure}
\begin{proof}By Lemma \ref{genus-framing}, since $\mathcal{F}(F)=0$, the genus of $F$ is even. We perform an ambient isotopy of $\mathbb{S}^3$ to obtain $F$ in disk-band form and then slide bands so that we get $F$ as a boundary connect sum of punctured Klein bottles where each punctured Klein bottle is of the form shown in Figure \ref{nosurface}, i.e.\ each has one orientation-preserving band. Of course, the bands may interact with each other in ways other than crossings, as shown in Figure \ref{examplecurve}, and bands of different summands may also interact. As in the remarks at the beginning of this section, $\mathcal{F}(F)$ can be computed by considering the different types of crossings between bands and the twists within each band, in some projection for $K$. Crossings between bands are of three types: crossings where both participating bands are orientation-preserving, crossings where one band is orientation-preserving and the other is orientation-reversing, and crossings where both participating bands are orientation-reversing. It is easy to calculate that the only non-zero contributions to $\mathcal{F}(F)$ are from crossings of the first type, i.e.\ where both participating bands are orientation-preserving, each of which contributes $\pm 4$ depending on the relative orientations of the crossing bands. Full twists of the orientation-preserving band can be deformed into crossings of this type and also contribute $\pm4$ depending on the `handednesss' of the twist. Since the two edges of the orientation-reversing bands have opposite orientations, twists in these bands do not contribute to $\mathcal{F}(F)$.

Consider $\gamma$, the curve which is the sum of the cores of the orientation preserving bands, as in shown in the example in Figure \ref{examplecurve}. $\lk(\gamma, \gamma^+)$ can be calculated by considering only the crossings and twists of the orientation-preserving bands, which are exactly the crossings that contribute to $\mathcal{F}(F)$. In fact, $\mathcal{F}(F) = 4\lk(\gamma,\gamma^+)$. Therefore, $\mathcal{F}(F)=0$ if and only if $\lk(\gamma,\gamma^+)=0$. By construction we see that this curve intersects each orientation-reversing curve on F transversely an odd number of times and each orientation-preserving curve an even number of times, which implies that it is the Poincar\'e dual of $w_1(F)$ and therefore, homologically essential. \end{proof}
%

\begin{prop}A 2-sided non-separating homologically essential simple closed curve on $F$, a punctured Klein bottle, is unique upto orientation and isotopy. \end{prop}
\begin{proof}There are exactly four isotopy classes of unoriented homologically essential simple closed curves on a Klein bottle \cite{Meeks79}\cite[Lemma 2.1]{Price77}. Moreover, any two 2-sided non-separating simple closed curves are isotopic (as unoriented curves) on the Klein bottle. Henceforth, the proof is much like Gilmer's proof of the corresponding fact about punctured tori in \cite{Gi93}. If we consider the isotopy on the Klein bottle, whenever the curve passes over the boundary component, we are effectively band-summing with the longitude of the Klein bottle. It is easily checked using a picture that band-summing a curve $\gamma$ with the longitude yields $\gamma$ with the opposite orientation. \end{proof}

We note that the curve $\gamma$ constructed in Proposition \ref{||curve} satisfies the conditions in the statement of the proposition above.

One reason for seeking curves of self-linking zero on a low genus Seifert surface is that one might perform surgery along it to reduce genus. The following result shows that the same is true for non-orientable surfaces.

\begin{prop}\label{genusdec}Given a connected non-orientable surface $F$ of genus $g$ and a single boundary component, surgering along a non-separating 2-sided curve $\gamma$ of zero self-linking, i.e.\ removing the annulus cobounded by two parallel copies of $\gamma$ and gluing in two disks, results in a disk if $g=2$. If the resulting surface is orientable, the genus is $\frac{g-2}{2}$; if the resulting surface is non-orientable, the genus is $g-2$. \end{prop}
\begin{proof} We know that $\chi(F)=1-g$. Note that removing an annulus from $F$ does not change the Euler characteristic, since $\chi(\text{annulus})=\chi(\mathbb{S}^1)=0$. Let $F'$ the final surface with genus $g'$. We have 
\begin{align*}
\chi(F') &=(1-g)+\chi(\text{2 disks}) -\chi(\text{2 circles})\\
&=(1-g)+2-0\\
&=3-g
\end{align*}
Since $\gamma$ is non-separating, $F'$ is connected. If $F'$ is non-orientable with genus $g'$, we have that $1-g'=3-g \Rightarrow g'=g-2$. If $F'$ is orientable with genus $g'$, we have that $1-2g'=3-g \Rightarrow g'=\frac{g-2}{2}$. \end{proof}

Note that if surgery is performed on the curve $\gamma$ dual to $w_1(F)$ constructed in Proposition \ref{||curve} the resulting surface is necessarily orientable---since every orientation-reversing curve on the original surface intersected $\gamma$ once, surgering along $\gamma$ effectively removes all orientation-reversing curves from $F$. 

The above proposition implies that if a knot $K$ has a surgery curve $\gamma$ which is a slice knot, $K$ is slice as well. It is easily seen that if $\gamma$ is additionally ribbon, $K$ is ribbon as well.

The following are basic results for knots with crosscap number 2 which do not appear in the literature and will be used in the proof of Corollary \ref{cableconv}. 

\begin{prop}\label{cables}Given any knots $K$ and $J$, the composite knot $K_{(2,p)}\#J_{(2,-p)}$ bounds a punctured Klein bottle $F$ with zero framing. There is a disk-band form for $F$ where the knot type of the orientation-preserving band is $K\#J$. \end{prop}
\begin{proof}$K_{(2,p)}$ and $J_{(2,-p)}$ bound M\"{o}bius bands with framing $2p$ and $-2p$ respectively, by the definition of the cabling operation. Taking the boundary connected sum of the two M\"obius bands gives us a punctured Klein bottle with zero framing. However, while the obtained surface is in disk-band form it does not have an orientation-preserving band (yet). We can obtain one by sliding one of the erstwhile M\"obius bands over the other. This results in an orientation preserving band whose core has the knot type $K\#J$.\end{proof}

The above proposition also implies that the $(2,1)$--cable of any knot $K$ bounds a punctured Klein bottle $F$ with zero framing, where the knot type of the orientation-preserving band of $F$ is $K$ (by letting $J$ be the unknot).

\section{Concordance invariants}

We recall the notation for infection on a knot, as described in \cite{COT04}. We start with a \textit{pattern knot} $R$, and an unknotted curve $\eta$ in $\mathbb{S}^3-R$ (the \textit{axis of infection}). Since $\eta$ is unknotted it bounds a disk. Tie all the strands of $R$ passing through this disk into some knot $J$, the \textit{infecting knot}. We make sure that any parallel strands being tied into $J$ have zero linking with one another. We obtain a knot as the result of infection and denote it by $R(\eta,J)$. It is easily seen that the above is an untwisted satellite operation. 

\begin{prop}\label{actuallyinfection} If a knot $K$ bounds a punctured Klein bottle $F$ with zero framing, then $K$ is smoothly concordant to a knot $R'=R(\eta,J)$, where $R$ is a ribbon knot bounding a punctured Klein bottle with zero framing, $J$ is the knot type of the core of the orientation-preserving band of $F$ given in disk-band form, and $\eta$ is a curve as shown in Figure \ref{kleinex}.\end{prop} 
\begin{figure}[ht!]
  \centering
  \includegraphics{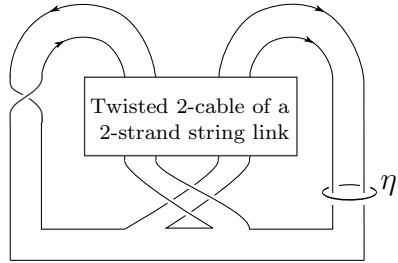}
  \put(-1.5,0.8){\tiny Twisted 2-cable of a}
  \put(-1.45,0.66){\tiny 2-strand string link}
  \put(0.02,0.4){$\eta$}
  \caption[]{Knot bounding a punctured Klein bottle with zero framing. The core of the orientation-preserving band is shown.}\label{kleinex}
\end{figure}
\begin{figure}[b]
  \centering
  \includegraphics{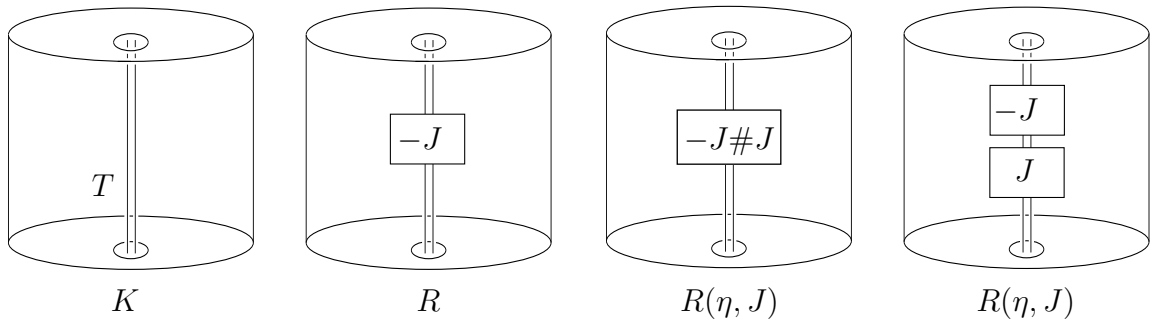}
  \put(-5.55,0.4){$T$}
  \put(-3.95,0.65){$-J$}
  \put(-2.45,0.65){$-J \# J$}
  \put(-0.83,0.8){$-J$}
  \put(-0.72,0.48){$J$}
  \put(-5.45,-0.2){$K$}
  \put(-3.85,-0.2){$R$}
  \put(-2.47,-0.2){$R(\eta,J)$}
  \put(-0.92,-0.2){$R(\eta,J)$}
  \caption[Proof of Proposition \ref{actuallyinfection}]{Proof of Proposition \ref{actuallyinfection}}\label{tangles}
\end{figure}
\begin{proof} We will follow the proof of Proposition 1.7 in \cite{CFrT09}. We isotope $F$ into a disk-band form with an orientation-preserving band, as in the proof of Proposition \ref{||curve}. We also know from Proposition \ref{||curve} that the core of the orientation-preserving band has zero self-linking. As a result, surgering along the core would give us a slice disk for the knot, as shown in Proposition \ref{genusdec}. Let $J$ denote the tangle whose closure is the knot type of the core of the orientation-preserving band of $F$. Notice that the orientation-preserving band can then be considered to be the (untwisted) 2-cable of $J$. Consider a curve $\eta$ linking once with the orientation-preserving band of $F$. It bounds a disk $E\subseteq \mathbb{S}^3$. If we thicken $E$ we get the local picture shown in Figure \ref{tangles}, where the orientation-preserving band appears as the 2-cable of the trivial tangle $T$. Replace the 2-cable of $T$ by the 2-cable of $-J$, and call the resulting knot $R$. Notice that this results in a new punctured Klein bottle, also with zero framing. The knot $R$ now bounds a punctured Klein bottle where the knot type of the orientation-preserving band is $-J \# J$, which is ribbon. By Proposition \ref{genusdec}, by surgering along the core of the orientation-preserving band, we see that the knot $R$ is also ribbon. 

Now consider the knot $R'$ obtained from $K$ by replacing $T$ by the 2-cable of $-J \# J$. Note that $R' = R(\eta, J)$, i.e.\ the infection of $R$ by $J$ along the curve $\eta$, by the equivalence of the last two panels of Figure \ref{tangles}. Since the trivial tangle is smoothly concordant to $-J \# J$, their 2-cables are also smoothly concordant. By modifying the trivial concordance from $K$ to itself by the tangle concordance between the 2-cables of the trivial tangle and $-J\#J$, we see that $R'=R(\eta,J)$ is smoothly concordant to $K$.\end{proof}

Recognizing that our knots are secretly infections provides a fair amount of information about certain knot invariants: 

\begin{prop} If a knot $K$ bounds a punctured Klein bottle $F$ with zero framing, and $J$ is the knot type of the core of the orientation-preserving band of any disk-band form for $F$, we have the following results:  
\begin{enumerate}
\item[a.] Arf $(K)=0$.
\item[b.] $\sigma_K(\omega)=\sigma_J(\omega^2)$, for all but finitely many $\omega$.
\item[c.] $K$ has (ordinary) signature $0$.
\item[d.] $|\tau(K)-2\tau(J)|\leq 4$.
\end{enumerate}
Here $\sigma_\cdot(\omega)$ denotes the Levine-Tristram signature function, and $\tau$ is the Floer homology invariant of Ozsv\'ath-Szab\'o \cite{OS03} and Rasmussen \cite{Ras03}.\end{prop}

\begin{proof}We know that Arf $(K)=0$ iff $\Delta_K(-1) \equiv \pm 1 \mod 8$ for any knot $K$ \cite{Mur69}, where $\Delta_K(t)$ is the Alexander polynomial. On the other hand, since $\lk(R,\eta)=2$,$$\Delta_{R(\eta,J)}(t)=\Delta_R(t)\Delta_J(t^2)$$ that is, $\Delta_{R(\eta,J)}(-1)=\Delta_R(-1)\Delta_J(1)$. We know that $\Delta_J(1)=\pm 1$. Since the Arf invariant is a concordance invariant, $R$ is ribbon, and $R(\eta,J)$ is smoothly concordant to $K$, we have that $\text{Arf} (R(\eta,J))=\text{Arf}(K)$ and $\Delta_R(-1)= \pm 1 \mod 8$. Part a.\ follows.

For Part b. we have from \cite{Li77, LivM85} that $$\sigma_{R(\eta,J)}(\omega)=\sigma_R(\omega)+\sigma_J(\omega^2)$$ since $\eta$ has winding number 2, for all $\omega$ except the roots of the Alexander polynomials of $R(\eta,J)$, $J$ and $R$. Since $R$ is ribbon, $\sigma_R(\omega)$ is the zero function, except at the roots of $\Delta_R(t)$. Part b, follows. Part c. follows as well by setting $\omega = -1$. 

We have from Theorem 1.2 in \cite{Rob12} that 
$$-n_+(R)-l \leq \tau(R(\eta,J)) - \tau(R) - l\tau(J) \leq n_+(R) + l$$ where $l=\lk(R,\eta)$ and $n_+(R)$ is the least number of positive intersections between $R$ and a disk bounded by $\eta$. In our case, we have $n_+(R)=l=2$ and since $R$ is smoothly slice, $\tau(R)=0$. Also, since $\tau$ is an invariant of smooth concordance, $\tau(R(\eta,J))=\tau(K)$. Therefore, $-4 \leq \tau(K)-2\tau(J)\leq 4$, proving Part d. \end{proof}

\begin{prop}If $K$ is slice and bounds a punctured Klein bottle $F$ with zero framing, then $J$, the knot type of the core of the orientation preserving band in any disk-band form for $F$, is 2-torsion in the algebraic knot concordance group.\end{prop}

\begin{proof}Let $\mathcal{AC}$ denote the algebraic knot concordance group, considered as the Witt group of nonsingular linking forms over certain torsion $\integers[t,t^{-1}]$-modules \cite{Kear75}. Given a knot $K$, the corresponding element of $\mathcal{AC}$ may be denoted by $(\mathcal{A}(K), \Bl(K))$, where $\mathcal{A}(K)$ is the Alexander module of $K$ and $\Bl(K)$ is the Blanchfield linking form. That is, if we denote algebraic concordance class by $[\,\cdot\,]$, $[K]= (\mathcal{A}(K),\Bl(K))$.

Consider the map $f:\mathcal{AC}\rightarrow \mathcal{AC}$, induced by $t \mapsto t^2$ (described in greater detail in \cite{Cha07}). We will show that $f([J]) = [R(\eta,J)]=[K]$, where $R$, $\eta$, $J$ are as in Proposition \ref{actuallyinfection}. We know from \cite{LivM85} that $$\mathcal{A}(R(\eta,J))=\mathcal{A}_0(R) \oplus \left(\mathcal{A}_0(J) \otimes_{\mathbb{Z}[t,t^{-1}]} W\right)$$ where $W=\mathbb{Z}[t,t^{-1}]$ as a $\mathbb{Z}[t,t^{-1}]$ module, where $t$ acts by $t\mapsto t^2$, since the winding number of $\eta$ is $2$. The map $t \mapsto t^2$ induces a similar transformation on the Blanchfield linking forms, that is, if $B_\cdot(t)$ is a matrix representing the Blanchfield linking form $$B_{R(\eta,J)}(t)=B_R(t) \oplus B_J(t^2)$$ We denote this new Blanchfield form as $\Bl(R(\eta,J))=\Bl(R)\oplus \left(\Bl(J) \otimes_{\mathbb{Z}[t,t^{-1}]} W\right)$ where $W$ is as above.

Since $R$ is a ribbon knot, $(\mathcal{A}(R),\Bl(R))$ is the zero Witt class in $\mathcal{AC}$. Therefore, $$(\mathcal{A}(R(\eta,J)),\Bl(R(\eta,J)))\cong \left(\mathcal{A}(J) \otimes_{\mathbb{Z}[t,t^{-1}]} W\, ,\, \Bl(J) \otimes_{\mathbb{Z}[t,t^{-1}]} W\right)$$ as claimed.

If the knot $R(\eta, J)$ is itself slice, we see that $\left(\mathcal{A}(J) \otimes_{\mathbb{Z}[t,t^{-1}]} W\, ,\, \Bl(J) \otimes_{\mathbb{Z}[t,t^{-1}]} W\right)$ is $0$ in $\mathcal{AC}$, i.e.\, $f([J])=0$. But we know from \cite[Proposition 2.1]{CO93}(See also \cite[Theorem 6]{ChaLivRu08}) that knots in the kernel of the map $f$ induced by $t \mapsto t^2$ must be $2$-torsion in $\mathcal{AC}$.\end{proof}

\section{Homology cobordism of zero-surgery manifolds}

Given a knot $K$, one frequently studies the associated 3-manifold, $M_K$, obtained by performing zero-framed surgery on $K$ in $\mathbb{S}^3$. Suppose the knots $K_0,K_1 \subseteq \mathbb{S}^3$ are concordant via an annulus $A\subseteq \mathbb{S}^3\times [0,1]$. By Alexander duality, the exterior of $A$ is a $\integers$-homology cobordism between the exteriors of $K_0$ and $K_1$. If we then adjoin a zero-framed $\mathbb{S}^1\times\mathbb{D}^2\times [0,1]$ to the homology cobordism between exteriors, we get a homology cobordism between $M_{K_0}$ and $M_{K_1}$. This observation has a converse when one of the knots is the unknot: 

\begin{prop}[Proposition 1.2 from \cite{CFHeHo13}] Suppose $K$ is any knot in $\mathbb{S}^3$ and  $U$ is the trivial knot. Then $M_K$ is smoothly homology cobordant to $M_U$ via a cobordism $V$ whose $\pi_1$ is normally generated by a meridian of $K$ if and only if $K$ bounds a smoothly embedded disk in a smooth manifold that is homeomorphic to $\mathbb{B}^4$. \end{prop}

This result gives us a way of translating information about zero-surgery manifolds to information about concordance relationships between knots. Here, we will use a related result for $R$-homology cobordisms: 
\begin{prop}[Proposition 1.5 from \cite{CFHeHo13}]\label{homcobtoslice} Suppose $K$ is any knot in $\mathbb{S}^3$ and $R\subseteq \mathbb{Q}$ is a non-zero subring. Let $U$ denote the trivial knot. Then $M_K$ is smoothly $R$-homology cobordant to $M_U$ if and only if $K$ is smoothly $R$-concordant to $U$ i.e.\ $K$ is smoothly $R$-slice.\end{prop}

In addition, recognizing that our knots are the result of infection allows us to use the following helpful theorem from \cite{CFHeHo13}:
\begin{thm}[Theorem 2.1 from \cite{CFHeHo13}] \label{cobordismthm} Suppose $R$ is a (smoothly) $\mathbb{Z}\left[\frac{1}{n}\right]$-slice knot, and $\eta$ is an unknotted curve with non-zero winding number $n$. Then, for any knot $J$, $M_J$ is smoothly $\mathbb{Z}\left[\frac{1}{n}\right]$-homology cobordant to $M_{R(\eta,J)}$.\end{thm}

The proof of the following result is an extension of the proofs of the results above to our context. 
\begin{thm} \label{cobordism proof}Suppose the knot $K$ is $\mathbb{Z}\left[\frac{1}{2}\right]$-slice and bounds a punctured Klein bottle $F$ with zero framing. Let $J$ be the knot type of the orientation-preserving band in any disk-band form for $F$. Then $J$ is smoothly $\mathbb{Z}\left[\frac{1}{2}\right]$-slice, i.e.\ $J$ bounds an embedded 2-disk in a 4-manifold $\mathcal{B}$ which is a $\mathbb{Z}\left[\frac{1}{2}\right]$-homology $\mathbb{B}^4$. 

In addition, if $K$ is smoothly slice, $\pi_1(\mathcal{B})$ is normally generated by a single element (the meridian of $K$), the meridian of $J$ is mapped to twice the generator of $H_1$ of the slice disk complement in $\mathcal{B}$, and the homology groups of $\mathcal{B}$ are as follows:
\begin{itemize}
\item $H_1(\mathcal{B};\mathbb{Z})=H_2(\mathcal{B};\mathbb{Z})=\mathbb{Z}/2$
\item $H_3(\mathcal{B};\mathbb{Z})=H_4(\mathcal{B};\mathbb{Z})=0$
\end{itemize}
\end{thm}

\begin{figure}[t]
  \centering
  \includegraphics{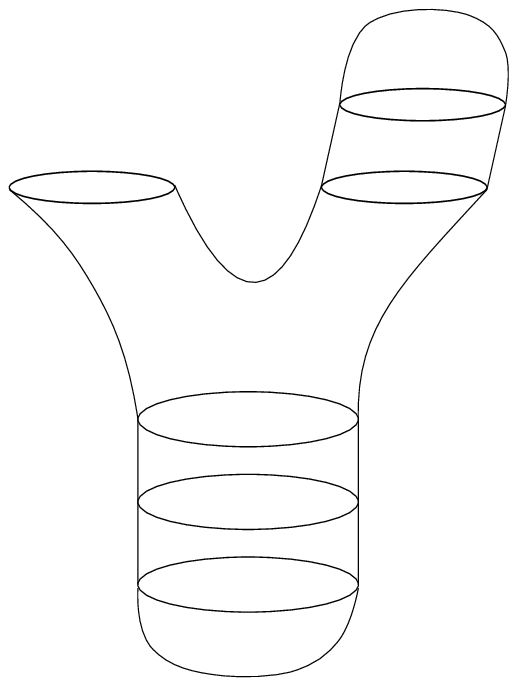}
  \put(-2.3,1.9){$M_J$}
  \put(0.02,2.25){$M_U$}
  \put(-0.05,1.9){$M_R$}
  \put(-0.55,1){$M_{R(\eta,J)}$}
  \put(-0.55,0.65){$M_K$}
  \put(-0.55,0.3){$M_U$}
  \caption[Proof of Theorem \ref{cobordism proof}]{Proof of Theorem \ref{cobordism proof}: Here $M_U\cong \mathbb{S}^1\times\mathbb{S}^2$'s have been capped off by $\mathbb{S}^1\times\mathbb{B}^3$'s.}\label{cobordism pic}
\end{figure}

\begin{proof}By Proposition \ref{actuallyinfection}, $K$ is smoothly concordant to $R(\eta,J)$, where $R$ is a ribbon knot and $J$ is as required. By the remarks at the beginning of this section, this gives us a $\integers$-homology cobordism between $M_K$ and $M_{R(\eta,J)}$. Since $R$ is smoothly slice and $\eta$ has winding number 2, Theorem \ref{cobordismthm} gives us a smooth $\mathbb{Z}\left[\frac{1}{2}\right]$-homology cobordism between $M_{R(\eta,J)}$ and $M_J$. Since $K$ is $\mathbb{Z}\left[\frac{1}{2}\right]$-slice, we have a $\integers\left[\frac{1}{2}\right]$-homology cobordism between $M_K$ and $M_U$, where $U$ is the unknot, by Proposition \ref{homcobtoslice}. By stacking the various cobordisms as in Figure \ref{cobordism pic}, we obtain that $M_J$ is smoothly $\integers\left[\frac{1}{2}\right]$-homology cobordant to $M_U$, and by Proposition \ref{homcobtoslice}, $J$ is smoothly $\integers\left[\frac{1}{2}\right]$-slice. This completes the proof of the first part of this theorem. 

To complete the proof, we need to take a closer look at the cobordism promised by Theorem \ref{cobordismthm}. Following the construction in \cite{CFHeHo13}, we have a cobordism between $M_R \sqcup M_J$ and $M_{R(\eta,J)}$, obtained as follows. Start with $M_J \times [0,1]$ and $M_R \times [0,1]$. Let $N(\eta)$ be a regular neighborhood of $\eta$ in $M_R$. We identify $N(\eta) \times \{1\}\subseteq M_R \times \{1\}$ with the surgery solid torus in $M_J \times \{1\}$ such that a parallel pushoff of $\eta$ is identified with the meridian of $K$. The resulting 4-manifold has boundary $M_J\sqcup M_R \sqcup -M_{R(\eta,J)}$. In addition, we have that $R$ is smoothly slice, and therefore, $M_R$ is homology cobordant to $M_U \cong \mathbb{S}^1\times\mathbb{S}^2$, which can be capped off by $\mathbb{S}^1\times\mathbb{B}^3$. This gives us the cobordism between $M_{R(\eta,J)}$ and $M_J$ claimed in Theorem \ref{cobordismthm} (the top half of Figure \ref{cobordism pic}). 

$R(\eta,J)$ is concordant to $M_K$, which gives us a homology cobordism between $M_{R(\eta,J)}$ and $M_K$.In addition, $K$ is slice, and therefore, we have a homology cobordism between $M_K$ and $M_U$. By gluing these cobordisms together we obtain a manifold with boundary $M_j\sqcup M_U$. Finally we cap off $M_U\cong \mathbb{S}^1\times\mathbb{S}^2$ by $\mathbb{S}^1\times \mathbb{B}^3$. This gives us a 4-manifold bounded by $M_J$, as shown in Figure \ref{cobordism pic}. We add a zero-framed 2-handle to $M_J$ along the meridian of $J$ to finally obtain the obtain the manifold $\mathcal{B}$ with $\partial\mathcal{B}=\mathbb{S}^3$, in which $J$ bounds a smoothly embedded disk, as desired.

The cobordism $W$ between $M_R$, $M_J$ and $M_{R(\eta,J)}$ deformation retracts to $M_{R(\eta,J)}\cup \eta\times \mathbb{B}^2$, so up to homotopy, we obtain the cobordism by adding a 2-cell and a 3-cell. The 2-cell is added along $\lambda_J$, the longitude of $J$ and $\pi_1(M_{R(\eta,J)})$ is normally generated by the meridian of $R(\eta,J)$. Therefore, $\pi_1(W)=\quot{\langle\langle\mu_{R(\eta,J)}\rangle\rangle}{\langle\langle\lambda_J\rangle\rangle}$, where $\langle\langle\cdot\rangle\rangle$ denotes normal closure in $\pi_1(W)$. 

The fundamental group of each of the other cobordisms is normally generated by the meridian of the relevant knot and the 2-handle added at the final stage kills off $\mu_J$, the meridian of the knot $J$. Putting this all together, we see that $$\pi_1(\mathcal{B})=\frac{\left(\quot{\langle\langle\mu_K\rangle\rangle}{\langle\langle\lambda_J\rangle\rangle}\right)}{\langle\langle\mu_J\rangle\rangle}$$
where $\langle\langle\cdot\rangle\rangle$ now denotes normal closure in $\pi_1(\mathcal{B})$. We note however, that $\lambda_J$ is contained in the normal closure of $\mu_J$ and therefore, $$\pi_1(\mathcal{B})=\quot{\langle\langle\mu_K\rangle\rangle}{\langle\langle\mu_J\rangle\rangle}$$

In particular, $\pi_1(\mathcal{B})$ is normally generated by $\mu_K$. Note that, in homology, $2\mu_K=\mu_J$ and hence, $H_1(\mathcal{B};\mathbb{Z})\cong\mathbb{Z}/2$. Since $\tilde{H_i}(\mathcal{B};\mathbb{Z}\left[\frac{1}{2}\right])=0$, $\tilde{H_i}(\mathcal{B};\mathbb{Z})$ is 2-torsion. We can recover all the other homology groups using the Universal Coefficient Theorem and Poincar\'{e}-Lefschetz Duality. All the homology groups below are with $\integers$ coefficients:

$$\mathbb{Z}/2 \cong H_1(\mathcal{B}) \cong H^3(\mathcal{B},\partial\mathcal{B})\cong \text{Hom}(H_3(\mathcal{B},\partial{\mathcal{B}}),\mathbb{Z})\oplus\text{Ext}(H_2(\mathcal{B},\partial{\mathcal{B}}),\mathbb{Z})$$
$$\Rightarrow \text{Ext}(H_2(\mathcal{B},\partial{\mathcal{B}}),\mathbb{Z}) \cong \mathbb{Z}/2$$
$$\Rightarrow \text{Torsion}(H_2(\mathcal{B},\partial{\mathcal{B}})) \cong \mathbb{Z}/2$$
Recall that $\partial{\mathcal{B}}=\mathbb{S}^3$. Therefore, using the homology exact sequence for a pair, we have:
$$0\rightarrow H_2(\mathcal{B})\xrightarrow{\cong} H_2(\mathcal{B},\partial{\mathcal{B}})\rightarrow 0 \rightarrow H_1(\mathcal{B})\xrightarrow{\cong} H_1(\mathcal{B},\partial{\mathcal{B}})\rightarrow 0$$
$$0\leftarrow H^2(\mathcal{B})\xleftarrow{\cong} H^2(\mathcal{B},\partial{\mathcal{B}})\leftarrow 0 \leftarrow H^1(\mathcal{B})\xleftarrow{\cong} H^1(\mathcal{B},\partial{\mathcal{B}})\leftarrow 0$$
As a result,
$$H_3(\mathcal{B})\cong H^1(\mathcal{B},\partial{\mathcal{B}})\cong H^1(\mathcal{B})\cong\text{Hom}(H_1(\mathcal{B}),\mathbb{Z})\cong 0$$
Since $H_2(\mathcal{B})$ is 2-torsion and $H_2(\mathcal{B})\cong H_2(\mathcal{B},\partial{\mathcal{B}})$ and $\text{Torsion}(H_2(\mathcal{B},\partial{\mathcal{B}})) \cong \mathbb{Z}/2$, $H_2(\mathcal{B})\cong\mathbb{Z}/2$.

We note that the slice disk $\Delta_J$ bounded by $J$ in the construction above is the co-core of the 2-handle added at the last stage and therefore, $\mu_J$ is mapped to twice the generator of $H_1(\mathcal{B}-\Delta_J)\cong \mathbb{Z}=\langle \mu_K\rangle$.\end{proof}

We should note that the condition of the meridian of $J$ mapping to twice the generator of the slice disk complement is related to the notion of being \textit{weakly} rationally slice \cite{Ka09}. In addition, we know that if $J$ is $\integers\left[\frac{1}{2}\right]$-slice, so is $K$, and therefore, we have actually proved that $K$ is $\integers\left[\frac{1}{2}\right]$-slice if and only if $J$ is $\integers\left[\frac{1}{2}\right]$-slice. Moreover, in conjunction with the remark following the proof of Proposition \ref{cables}, we have now proved: 
\begin{cor}For a knot $K$ if the (2,1)--cable is slice, or even just $\integers\left[\frac{1}{2}\right]$-slice, then K is $\mathbb{Z}\left[\frac{1}{2}\right]$-slice.\end{cor}

We are also now able to prove the following:
\begin{cor}\label{cableconv}Given knots $K$ and $J$, if $K_{(2,p)}$ is $\integers\left[\frac{1}{2}\right]$-concordant to $J_{(2,p)}$, then $K$ is $\mathbb{Z}\left[\frac{1}{2}\right]$-concordant to $J$. In particular, if $K_{(2,p)}$ is concordant to the $(2,p)$--torus knot, then $K$ is $\mathbb{Z}\left[\frac{1}{2}\right]$-slice.\end{cor}
\begin{proof}First we note that $-\left(J_{(2,p)}\right)=(-J)_{(2,-p)}$. We know from Proposition \ref{cables} that $K_{(2,p)} \# -\left(J_{(2,p)}\right) = K_{(2,p)} \# (-J)_{(2,-p)}$ bounds a punctured Klein bottle with zero framing, where we may consider $K\#-J$ to be the knot type of the orientation-preserving band. Since $K_{(2,p)}$ is $\integers\left[\frac{1}{2}\right]$-concordant to $J_{(2,p)}$, $K_{(2,p)} \# -\left(J_{(2,p)}\right)$ is $\integers\left[\frac{1}{2}\right]$-slice, and we are in the situation of Theorem \ref{cobordism proof}. Therefore, $K\#-J$ is $\integers\left[\frac{1}{2}\right]$-slice, and so $K$ is $\integers\left[\frac{1}{2}\right]$-concordant to $J$.\end{proof}

\bibliographystyle{abbrv}
\bibliography{knotbib}
\end{document}